\documentclass[12pt]{article}
 \usepackage{a4}
 \usepackage{amsmath}
 \usepackage{amssymb}
\usepackage{amsthm}
\usepackage{graphicx}
\newtheorem{thm}{Theorem}[section]
\newtheorem{cor}[thm]{Corollary}
\newtheorem{lem}[thm]{Lemma}
\newtheorem{prop}[thm]{Proposition}

\newtheorem{rem}[thm]{Remark}

\renewcommand{\a}{\alpha}
\renewcommand{\b}{\beta}
\newcommand{\g}{\gamma}
\def\eps{\varepsilon}
\newcommand{\supp}{\operatorname{supp}}

\def\C{\mathbb{C}}
\def\R{\mathbb{R}}
\def\N{\mathbb{N}}

\title{Symmetric moment problems and a conjecture of Valent}

\author{C. Berg and R. Szwarc\footnote{The author acknowledges support by Polish grant NCN 2013/11/B/ST1/02308}}

\begin{document}
\maketitle

\begin{abstract} In 1998 G. Valent made conjectures about the order and type of certain indeterminate Stieltjes moment problems associated with birth and death processes  having polynomial birth and death rates of degree $p\ge 3$. Romanov recently proved  that the order is $1/p$ as conjectured, see \cite{Ro}. We prove that the type with respect to the order is  related to  certain multi-zeta values and that this type belongs to the interval 
$$[\pi/(p\sin(\pi/p)),\pi/(p\sin(\pi/p)\cos(\pi/p))],
$$ 
which also contains the conjectured value. This proves that the conjecture about type is asymptotically correct as $p\to\infty$.

The main idea is to obtain estimates for order and type of symmetric indeterminate Hamburger moment problems when the orthonormal polynomials $P_n$ and those of the second kind $Q_n$ satisfy $P_{2n}^2(0)\sim c_1n^{-1/\b}$ and $Q_{2n-1}^2(0)\sim c_2 n^{-1/\a}$, where
$0<\a,\b<1$ can be different, and $c_1,c_2$ are positive constants. In this case the order of the moment problem is majorized by the harmonic mean of $\a,\b$.  
Here $\alpha_n\sim \beta_n$ means that $\alpha_n/\beta_n\to 1$. This also leads to a new proof of Romanov's Theorem that the order is $1/p$.
\end{abstract}



\section{Introduction and results} To every indeterminate moment problem is associated an order $\rho\in[0,1]$, namely the common order of the four entire functions in the Nevanlinna matrix. If  $0<\rho<1$ there is also associated a type $\tau\in[0,\infty]$ which is the common type of the same functions, see \cite{B:P}. When $\rho=1$ the type is necessarily $\tau=0$ by a theorem of M. Riesz.

Since a moment problem is characterized by two sequences $(a_n)_{n\ge 0}$ and $(b_n)_{n\ge 0}$ of respectively real and positive numbers (called the Jacobi parameters or recurrence coefficients) via the three term recurrence relation
\begin{equation}\label{eq:3term}
z r_n(z)=b_nr_{n+1}(z)+a_nr_n(z)+b_{n-1}r_{n-1}(z),\quad n\ge 0,
\end{equation}
satisfied by the orthonormal polynomials $P_n(z)$ and those of the second kind $Q_n(z)$, it is of some interest to be able to calculate the order and type directly from the sequences $a_n,b_n$ without calculating the functions of the Nevanlinna matrix. For $P_n$ we have the initial conditions
$P_{-1}=0,P_0=1$, while we have $Q_{-1}=-1,Q_0=0$ with the convention that $ b_{-1}:=1$.
 
Such results were obtained about order in the paper \cite{B:S} by the authors. In particular we obtained the following result:

\begin{thm}\label{thm:beralpha1} Assume that the coefficients $a_n,b_n$  satisfy 
\begin{equation}\label{eq:finite}
\sum_{n=1}^\infty \frac{1+|a_n|}{\sqrt{b_nb_{n-1}}}<\infty,
\end{equation}
and that $(b_n)$ is either eventually log-convex or eventually log-concave. 

Then the order  $\rho$ of the moment problem is equal to the exponent of convergence $\mathcal E(b_n)$of the sequence $(b_n)$, where
$$
\mathcal E(b_n)=\inf \left\{\alpha>0 \mid \sum_{n=0}^\infty
b_n^{-\alpha}<\infty \right\}.
$$
\end{thm}

This result is well suited for applications to symmetric moment problems characterized by $a_n=0$ for all $n$.
In this case and when the log-convexity/log-concavity condition  is satisfied, the condition \eqref{eq:finite} is equivalent to $\sum 1/b_n<\infty$, which in turn is equivalent to indeterminacy.

However, if the log-convexity/log-concavity is not satisfied, then the condition $\sum 1/b_n<\infty$ can hold also for determinate moment problems, see e.g. \cite[Remark 4.5]{B:S}.

It turns out that for a number of the classical indeterminate Stieltjes moment problems, where $a_n>0$, the condition \eqref{eq:finite} is not satisfied. This is in particular true for the Stieltjes moment problem associated with the birth and death rates \eqref{eq:birth} and \eqref{eq:death}, where $(1+|a_n|)/\sqrt{b_nb_{n-1}}$ converges to 2 for $n\to\infty$. 
To handle these cases we need some refinements of our results. 

The idea is to use the well-known one-to-one correspondence between Stieltjes moment problems and symmetric Hamburger moment problems as recalled in Section 3. If the order and type of an indeterminate symmetric Hamburger moment problem can be calculated, then the order and type of the corresponding Stieltjes moment problem are known too, see Proposition~\ref{thm:order/type} for details.

Before announcing our first main result we recall that a sequence of complex numbers $(x_n)$ belongs to $\ell^\a$, where $\a>0$, if $\sum |x_n|^\a<\infty$. The result reads:

\begin{thm}\label{thm:main1} Consider a symmetric indeterminate Hamburger moment problem and introduce the notation
\begin{equation}\label{eq:sym1}
v_n=P_{2n}^2(0),\quad u_n=Q_{2n-1}^2(0),\quad n\ge 1.  
\end{equation}
Assume that 
\begin{enumerate}
\item[\rm{(i)}] $(u_n)\in\ell^\alpha,\;(v_n)\in\ell^\beta$, where $0<\alpha,\beta\le 1$.
\end{enumerate}
Assume also that there exist constants $C,D>0$ such that
\begin{enumerate}
\item[\rm{(ii)}] $u_n\le Cv_n^{\beta/\alpha},\; n\ge 1$.
\item[\rm{(iii)}] $v_j\le Dv_i$ for $i\le j$.
\end{enumerate}
Let $\gamma$ be the harmonic mean of $\alpha$ and $\beta$, i.e.,
\begin{equation}\label{eq:harm}
\gamma^{-1}=\frac12\left(\alpha^{-1}+\beta^{-1}\right).
\end{equation}
Then the order $\rho$ of the moment problem is $\le\gamma$.
\end{thm}

The proof of Theorem~\ref{thm:main1} is given in Section 5.

\begin{rem}\label{thm:symm} {\rm Note that condition (iii) is fulfilled, if $(v_n)$ is eventually decreasing.

As the proof will show, a similar result is true if the conditions (ii) and (iii)  are replaced by the following  conditions, where the roles of $u_n$ and $v_n$ are interchanged
\begin{enumerate}
\item[\rm{(ii)}'] $v_n\le C' u_n^{\a/\b},\; n\ge 1$,
\item[\rm{(iii)'}] $u_j\le D' u_i$ for $i\le j$.
\end{enumerate}
Here $C',D'>0$ are suitable constants.
}
\end{rem}

Birth and death processes lead to Stieltjes moment problems according to a theory developed by Karlin and McGregor, see \cite{K:M} and \cite{I:L:M:V}. The theory depends on two sequences $(\lambda_n),(\mu_n)$ of birth and death rates with the constraints $\lambda_n,\mu_{n+1}>0,n\ge 0$ and $\mu_0\ge 0$. For simplicity we shall always assume $\mu_0=0$. 

The recurrence coefficients $(a_n),(b_n)$ of the corresponding Stieltjes moment problem are given by
\begin{equation}\label{eq:bdtoS}
a_n=\lambda_n+\mu_n,\quad b_n=\sqrt{\lambda_n\mu_{n+1}},\quad n\ge 0.
\end{equation}	

In a series of papers Valent and his co-authors studied Stieltjes moment problems coming from birth and death processes with polynomial rates, see
\cite{Va06} and references therein. In \cite{Va98} Valent formulated some conjectures in case of polynomial birth and death rates $(\lambda_n),(\mu_n)$ of the form 

\begin{equation}\label{eq:birth}
\lambda_n=(pn+e_1)\ldots(pn+e_p), \quad n\ge 0,
\end{equation}

and
\begin{equation}\label{eq:death}
\mu_n=(pn+d_1)\ldots(pn+d_p), \quad n\ge 0, 
\end{equation}
where it is assumed that
\begin{equation}\label{eq:natcond}
 0<e_1\le e_2\le \ldots\le e_p,\quad -p<d_1\le d_2\le\ldots\le d_p,\quad d_1\ldots d_p=0,
\end{equation}
so one  has $\lambda_n,\mu_{n+1}>0$ for $n\ge 0$ and $\mu_0=0$.

By introducing the quantities
\begin{equation}\label{eq:DE}
E=e_1+\cdots+e_p,\quad D=d_1+\cdots+d_p,
\end{equation}
Valent proved that the Stieltjes problem is indeterminate if and only if
\begin{equation}\label{eq:valindet}
1<\frac{E-D}{p}< p-1.
\end{equation}
This will also follow from the considerations in Section 4.

 In the special cases $p=3,4$ and for special values of $e_j,d_j,j=1,\ldots,p$, it has been possible to calculate the Nevanlinna matrices using elliptic functions, see \cite{B:V},\cite{G:L:V},\cite{G:L:R:V},\cite{Va06}. On the basis of these calculations Valent \cite{Va98} formulated the following conjecture:

{\bf Valent's conjecture} {\it For the indeterminate Stieltjes moment problem with the rates \eqref{eq:birth} and \eqref{eq:death} satisfying the condition 
\eqref{eq:valindet}
(hence $p\ge 3$), the order, type and Phragm{\'e}n-Lindel{\"o}f indicator function are given as
\begin{equation}\label{Valentconj}
\rho=1/p,\quad \tau=\int_0^1\frac{du}{(1-u^p)^{2/p}},\quad h(\theta)=\tau\cos((\theta-\pi)/p),\;\theta\in[0,2\pi].
\end{equation} 
}

Recently, Romanov \cite{Ro} has proved Valent's conjecture concerning order using the powerful theory of canonical systems.

\begin{thm}[Romanov \cite{Ro}]\label{thm:valent-order} The  indeterminate Stieltjes moment pro\-blem
with the rates \eqref{eq:birth} and \eqref{eq:death} satisfying \eqref{eq:valindet} has order $\rho=1/p$.
\end{thm}

Based on Theorem~\ref{thm:main1} we can give a new proof of Theorem~\ref{thm:valent-order}. Our method also leads to an estimate for the type of the moment problem in the previous theorem.
The methods in \cite{Ro} do not immediately lead to an upper bound for the type. Our estimates of the type depend on the following quantity.

 For real $p>2$ and $n\ge 1$ define the multi-zeta value\footnote{Observe that the inequalities between the indices $k_j$ are alternating between $\le$ and $<$: $k_{2j-1}\le k_{2j}<k_{2j+1}\le k_{2j+2}$}
\begin{equation}\label{eq:mz}
\gamma_n(p)=\sum_{1\le k_1\le k_2<\ldots<k_{2n-1}\le k_{2n}}\left(k_1k_2\ldots k_{2n-1}k_{2n}\right)^{-p/2}.
\end{equation}

\begin{thm}\label{thm:main2}
The  type $\tau$ of the indeterminate Stieltjes moment problem
with the rates \eqref{eq:birth} and \eqref{eq:death} satisfying \eqref{eq:valindet} is equal to $\tau_p/p$, where $\tau_p$ is the type of the entire function
\begin{equation}\label{eq:Gp}
G_p(z)=\sum_{n=1}^\infty \gamma_n(p) z^n
\end{equation}   
with coefficients $\gamma_n(p)$ given by \eqref{eq:mz}.
\end{thm}

\begin{rem}\label{thm:mainrem} {\rm  The function \eqref{eq:Gp} has order $1/p$ like the Stieltjes problem. This follows from the proof of Theorem~\ref{thm:main2}.}
\end{rem}

\begin{thm}\label{thm:valenttype} The  type $\tau$ of the indeterminate Stieltjes moment problem with the rates \eqref{eq:birth} and \eqref{eq:death} satisfying \eqref{eq:valindet} fulfils the inequalities
\begin{equation}\label{eq:type}
\frac{\pi}{p\sin(\pi/p)}\le \tau\le \frac{\pi}{p\sin(\pi/p)\cos(\pi/p)}.
\end{equation} 
\end{thm}

We have not been able to prove that 
$$
\tau=\int_0^1\frac{du}{(1-u^p)^{2/p}}
$$
as conjectured by Valent, but we prove in Section 6 that the value of the integral lies
between the bounds of \eqref{eq:type}. This shows that the conjecture about type is asymptotically correct as $p\to\infty$.

There is a modification of the multi-zeta value \eqref{eq:mz} which also determines the type
of the moment problem in question. The modified multi-zeta value in Theorem~\ref{thm:main3} has the advantage in comparison with $\gamma_n(p)$ in \eqref{eq:mz} that the indices of summation are strictly increasing as in the classical multiple zeta values, cf. \cite{Z}.

\begin{thm}\label{thm:main3} For real $p>2$ and $n\ge 2$ define
\begin{equation}\label{eq:mz1}
\zeta_n(p)= \sum_{1\le k_1<k_2<\ldots<k_n}(k_2-k_1)(k_3-k_2)\ldots(k_n-k_{n-1})
(k_1k_2\ldots k_n)^{-p}.
\end{equation}
Then we have
\begin{equation}\label{eq:mz3}
\zeta_n(p) \le \gamma_n(p)\le
\zeta(p-1)\frac{2^{p/2-1}}{p/2-1}\zeta_{n-1}(p),\quad n\ge 3,
\end{equation}
where
$$
\zeta(x)=\sum_{k=1}^\infty k^{-x},\quad x>1
$$
is the Riemann zeta-function.
\end{thm}

\begin{cor} For real $p>2$ the function 
\begin{equation}\label{eq:Zp}
Z_p(z)=\sum_{n=2}^\infty \zeta_n(p) z^n
\end{equation}
has the same order $1/p$ and the same type $\tau_p$ as the function $G_p$ given in \eqref{eq:Gp}.
\end{cor}

\begin{rem} {\rm The previous results show that the order and type of the indeterminate Stieltjes moment problem associated with the rates \eqref{eq:birth} and \eqref{eq:death} are  independent of the
special values of the parameters $e_j,d_j$. It is therefore enough to consider the special case where $e_j=p/2, d_j=0,j=1,\ldots,p$. In this case $(E-D)/p=p/2$ so the indeterminacy condition \eqref{eq:valindet} holds when $p\ge 3$. In this case the recurrence coefficients $(b_n)$ of the corresponding symmetric indeterminate Hamburger problem
are given as
$$
b_n=(p/2)^{p/2}(n+1)^{p/2}.
$$
}
\end{rem}

In the formula for $b_n$ above we can let $p$ can be any positive real number, and the corresponding Hamburger moment problem is indeterminate if and only if $p>2$, cf. \cite[Example 4.12]{B:S}.
This leads to the following result:

\begin{thm}\label{thm:last} Let $c>1$. The indeterminate symmetric Hamburger moment problem associated with $b_n=n^c$ for $n\ge 1$ has order $1/c$ and type $T_c/2$, where $T_c$ is the common type of the functions $G_{2c}(z)$ and $Z_{2c}(z)$.
\end{thm}

Using \eqref{eq:deftype} we have
\begin{equation}\label{eq:Tc1}
T_c=\frac{c}{e}\limsup_{n\to\infty}\left(n(\g_n(2c))^{1/(cn)}\right)=
\frac{c}{e}\limsup_{n\to\infty}\left(n(\zeta_n(2c))^{1/(cn)}\right),
\end{equation}
but we do not know how to evaluate $T_c$.
In accordance with Valent's conjecture we believe that $T_c$ is given in terms of the Beta-function as
\begin{equation}\label{eq:Tc2}
T_c=B(1/(2c),1-1/c).
\end{equation}

\section{Preliminaries} 

  Let $\mathbb M^\ast(\R)$ denote the set of positive measures on
$\R$ with moments of any order and infinite support. For $\mu\in\mathbb M^\ast(\R)$ we define the moment sequence
\begin{equation}\label{eq:Hmom} 
s_n=s_n(\mu)=\int_{-\infty}^\infty x^n\,d\mu(x),\quad n\ge 0.
\end{equation}
It is called normalized if $s_0=\mu(\R)=1$.
The theory of the indeterminate moment problem is treated in the classical monographs \cite{Ak}, \cite{S:T} and in the survey paper \cite{Si}, which introduced important new ideas. Here we follow the terminology and notation from Akhiezer  \cite{Ak}.  We shall also rely on concepts and notation introduced in our paper \cite{B:S}. The following is a classical result, which follows from proofs in \cite[Chap.1, Sect. 3]{Ak}: 

\begin{thm}\label{thm:indet} For $(s_n)$ as in \eqref{eq:Hmom} the
  following conditions are equivalent:
\begin{enumerate}
\item[\rm{(i)}]$\sum_{n=0}^\infty \left(P_n^2(0)+Q_n^2(0)\right)<\infty,$
\item[\rm{(ii)}] $ P(z)=\left(\sum_{n=0}^\infty |P_n(z)|^2\right)^{1/2}<\infty,\quad z\in\mathbb C.$
\end{enumerate}
If {\rm (i)} and {\rm (ii)} hold (the indeterminate case), then 
$Q(z)=\left(\sum_{n=0}^\infty |Q_n(z)|^2\right)^{1/2}<\infty$ for $z\in\mathbb C$,
 and the series for $P,Q$ are uniformly convergent on compact subsets of $\C$.
\end{thm}

We now consider a normalized indeterminate moment sequence $(s_n)_{n\ge 0}$.
 
The following polynomials will be used, cf. \cite[p.14]{Ak}
\begin{eqnarray*}
 A_n(z)&=& z\sum_{k=0}^{n-1} Q_k(0)Q_k(z),\\
B_n(z)&=& -1+z\sum_{k=0}^{n-1} Q_k(0)P_k(z),\\
C_n(z)&=& 1+z\sum_{k=0}^{n-1} P_k(0)Q_k(z),\\
D_n(z)&=& z\sum_{k=0}^{n-1} P_k(0)P_k(z).
\end{eqnarray*}
Because of Theorem~\ref{thm:indet} these polynomials tend to entire functions denoted $A,B,C,D$,  when $n$ tends to infinity.

Following ideas of Simon \cite{Si}, we can write the polynomial equations as
\begin{multline}
\begin{pmatrix}
A_{n+1}(z) & B_{n+1}(z)\\
C_{n+1}(z) & D_{n+1}(z)
\end{pmatrix} =\\
\biggl [ I + z
\begin{pmatrix}
-P_n(0)Q_n(0) & Q_n^2(0)\\
-P_n^2(0) & P_n(0)Q_n(0)
\end{pmatrix}
\biggr] \begin{pmatrix}
A_n(z) & B_n(z)\\
C_n(z) & D_n(z)
\end{pmatrix},
\end{multline}
hence
$$
\begin{pmatrix}
A_{n+1}(z) & B_{n+1}(z)\\
C_{n+1}(z) & D_{n+1}(z)
\end{pmatrix} =(I+zT_n)(I+zT_{n-1})\ldots(I+zT_0)
\begin{pmatrix}
 0 & -1\\
1 & 0
\end{pmatrix},
$$
where
\begin{equation}\label{eq:Tn}
T_n=\begin{pmatrix}
-P_n(0)Q_n(0) & Q_n^2(0)\\
-P_n^2(0) & P_n(0)Q_n(0)
\end{pmatrix}.
\end{equation}
Letting $n\to\infty$ we get
\begin{equation}\label{eq:Simon}
\begin{pmatrix}
A(z) & B(z)\\
C(z) & D(z)
\end{pmatrix} =\left[\prod_{n=1}^{\infty} (I+zT_n)\right]
\begin{pmatrix}
1 & 0\\
-z & 1
\end{pmatrix}
\begin{pmatrix}
 0 & -1\\
1 & 0
\end{pmatrix},
\end{equation}
where matrix products are supposed to expand to the left. The matrix on the left of \eqref{eq:Simon} is called the Nevanlinna matrix of the indeterminate moment problem.

For a non-constant entire function
\begin{equation}\label{eq:entire}
f(z)=\sum_{n=0}^\infty c_nz^n
\end{equation}
the maximum modulus is the increasing function
\begin{equation}\label{eq:mm}
\mathcal M_f(r)=\max_{|z|=r}|f(z)|,\quad r\ge 0.
\end{equation}
The order $\rho=\rho_f$ of $f$ is given  by the formulas
\begin{equation}\label{eq:deforder}
\rho_f=\limsup_{n\to\infty}\frac{\log n}{\log(\frac{1}{\root{n}\of{|c_n|}})}
=\limsup_{r\to\infty}\frac{\log\log\mathcal M_f(r)}{\log r}.
\end{equation}
If $0<\rho<\infty$ the type $\tau=\tau_f$ is given as
\begin{equation}\label{eq:deftype}
\tau_f=\frac{1}{e\rho}\limsup_{n\to\infty}(n|c_n|^{\rho/n})=\limsup_{r\to\infty}\frac{\log\mathcal M_f(r)}{r^\rho}.
\end{equation}
If $0<\tau<\infty$ we define the Phragm{\'e}n-Lindel{\"o}f indicator function
\begin{equation}\label{eq:PL}
h_f(\theta)=\limsup_{r\to\infty}\frac{\log|f(re^{i\theta})|}{r^\rho},\quad\theta\in\R.
\end{equation}
For details see \cite{BOAS}.

 For two entire functions $f,g$ it is easy to see that if 
\begin{equation}\label{eq:modineq}
\mathcal M_f(r)\le r^c\mathcal M_g(Kr)
\end{equation}
for $r$ sufficiently large, and $c,K$ are suitable constants, then $\rho_f\le\rho_g$. If in addition 
$0<\rho=\rho_f=\rho_g<\infty$, then $\tau_f\le K^\rho \tau_g$. 

Changes of finitely many of the parameters $(a_n),(b_n)$ do not change the indeterminate moment problem in an essential way as  expressed in the following classical result:

\begin{prop}\label{thm:finch} Consider an indeterminate moment problem corresponding to sequences $(a_n),(b_n)$ from \eqref{eq:3term}. If another moment problem is given in terms of $(\tilde a_n), (\tilde b_n)$, and if $\tilde a_n=a_n,\tilde b_n=b_n$ for $n\ge n_0$, then the second problem is also indeterminate and the two problems have the same order, type and 
 Phragm{\'e}n-Lindel{\"o}f indicator function.
\end{prop}

For a proof one can refer to results about the abbreviated Jacobi matrix, cf. \cite[p. 28]{Ak} given by the sequences $(a_{n+1}), (b_{n+1})$. The problem corresponding to $(a_{n+1}), (b_{n+1})$ is indeterminate with the same order, type and  Phragm{\'e}n-Lindel{\"o}f indicator function as the original problem. This follows from simple linear algebra considerations about solutions to difference equations of order 2. 
A complete proof  can be found in \cite{Pe}.

We end with a simple but useful result about how the order, type and indicator function change, when the Jacobi 
parameters are multiplied by a constant $c>0$.

\begin{prop}\label{thm:Jacobichange} Consider an indeterminate Hamburger moment problem of order $\rho$ and type $\tau$ corresponding to the sequences $(a_n),(b_n)$ from the three term recurrence relation. For $c>0$ the moment problem corresponding to the sequences $(ca_n),(cb_n)$ is also indeterminate with order $\rho(c)=\rho$, type $\tau(c)=\tau/c^\rho$ and indicator
$h(c)(\theta)=h(\theta)/c^\rho$.  
\end{prop}

\begin{proof} Let $(P_n)$ and $(P_n(\cdot;c))$ denote the orthonormal polynomials corresponding to the two sets of Jacobi parameters. Then it is easy to see that $P_n(x;c)=P_n(x/c)$. This means that the $D$-functions from the Nevanlinna matrices satisfy $D(x;c)=cD(x/c)$, and from this the relations follow. 
\end{proof}

\section{Symmetric Hamburger moment problems versus Stieltjes moment problems and birth and death processes}

In this section we have collected some well-known facts about the relation between Stieltjes moment problems and symmetric Hamburger moment problems. For details see \cite{Ch} and \cite{Be95}. 

Let 
 $$\mathbb M^\ast(\R_+)=\{\sigma\in\mathbb M^\ast(\R)\mid\supp(\sigma)\subseteq[0,\infty[\,\}\;.
 $$

A non-degenerate Stieltjes moment sequence is a sequence of the form
\begin{equation}\label{eq:St}
s_n=\int_0^\infty x^nd\sigma(x),\quad n\ge 0,
\end{equation}
where $\sigma\in\mathbb M^\ast(\R_+)$. It is called normalized if $s_0=1$.

 The sequence $(s_n)$ can be determinate or indeterminate in the
sense of Stieltjes, denoted det(S) and indet(S) respectively,
meaning that there is exactly one measure or more than one
measure from $\mathbb M^\ast(\R_+)$ satisfying \eqref{eq:St}.

In the study of a Stieltjes moment problem it is useful to consider
an accompanying symmetric Hamburger moment problem. Let $\mathbb
M_s^\ast(\R)$ denote the set of {\it symmetric} measures
$\mu\in\mathbb M^\ast(\R)$, i.e.,  $\mu(-B)=\mu(B)$ for Borel sets
$B\subseteq\R$. 

The map $\psi(x)=x^2$ induces a bijection of
$\mathbb M_s^\ast(\R)$ onto $\mathbb M^\ast(\R_+)$ given by
$\sigma=\psi(\mu)$ (the image measure of $\mu$ under $\psi$), i.e., if
$f:\R_+\to\C$ is a bounded Borel function then
 $$\int f(t)d\sigma(t)=\int f(x^2)d\mu(x).
 $$

A non-degenerate symmetric  Hamburger moment sequence is of the form
\begin{equation}\label{eq:Ha}
t_n=\int_{-\infty}^\infty x^n\,d\mu(x),\quad n\ge 0,
\end{equation}
where $\mu\in\mathbb M_s^\ast(\R)$. Clearly the odd moments vanish and $t_{2n}=s_n$ from \eqref{eq:St}, where $\sigma=\psi(\mu)$. Conversely, any non-degenerate Stieltjes moment sequence \eqref{eq:St} arises in this way from a unique symmetric Hamburger moment sequence.

The Stieltjes moment problem is indet(S)  if and only if the corresponding symmetric Hamburger problem is indeterminate. In case of indeterminacy there is a simple relation between the Nevanlinna matrices of the two problems. Let us just mention the following relation between the "$C$ "-functions $C_s$ and $C$ from the Nevalinna matrix \eqref{eq:Simon}, namely: $C_s(z)=C(z^2)$, where $C_s$ refers to the symmetric Hamburger problem
and $C$ to the Stieltjes problem. This yields the following result:

\begin{prop}\label{thm:order/type} In the indeterminate case let $\rho,\tau,h$ respectively $\rho_s,\tau_s,h_s$ denote the order, type and Phragm\'en-Lindel\"of indicator function of the Stieltjes problem respectively symmetric Hamburger problem. Then we have
$$
\rho=\rho_s/2,\; \tau=\tau_s,\, h(\theta)=h_s(\theta/2). 
$$
\end{prop}

Let us recall the connection between the three term recurrence relation (in the notation of \cite{Ak})
\begin{equation}\label{eq:3t}
z P_n(z)=b_nP_{n+1}(z)+a_nP_n(z)+b_{n-1}P_{n-1}(z),\quad n\ge 0,
\end{equation}
for a Stieltjes moment problem, where $a_n,b_n>0$ for $n\ge 0$,
and the three term recurrence relation
\begin{equation}\label{eq:sym3t} 
zS_n(z)=\b_nS_{n+1}(z)+\b_{n-1}S_{n-1}(z)
\end{equation}
for the corresponding symmetric Hamburger moment problem.
We have
\begin{equation}\label{eq:consym}
a_n=\b_{2n}^2+\b_{2n-1}^2,n\ge 1,\;a_0=\b_0^2,\quad
b_n=\b_{2n}\b_{2n+1},n\ge 0.
\end{equation}
 From this equation it is easy to calculate $(a_n),(b_n)$ from $(\b_n)$. Conversely, if $(a_n),(b_n)$ are given, then $(\b_n)$ is uniquely determined by the same equation.

 If we compare \eqref{eq:consym}  and \eqref{eq:bdtoS}, we see that any Stieltjes moment problem comes from a birth and death process with rates
$$
\lambda_n=\b_{2n}^2,\; \mu_{n+1}=\b_{2n+1}^2,\;n\ge 0,\quad \mu_0=0.
$$
\medskip

From this we get
\begin{prop}\label{thm:sumup}
The symmetric Hamburger moment problem corresponding to the Stieltjes problem associated with the birth and death rates \eqref{eq:birth} and \eqref{eq:death} has  the recurrence coefficients
\begin{equation}\label{eq:bn}
 b_{2n}=\sqrt{(pn+e_1)\ldots(pn+e_p)},\quad b_{2n+1}=\sqrt{(p(n+1)+d_1)\ldots(p(n+1)+d_p)}.
\end{equation}
\end{prop}

\section{Some technical Lemmas}

We are going to analyse the symmetric moment problem of Proposition~\ref{thm:sumup} and for this we need some lemmas. By $\a_n\sim\b_n$ is meant that $\a_n/\b_n\to 1$ for $n\to\infty$.

\begin{lem}\label{thm:deltan} For $x>-1$ and $n\in\mathbb N$ define
\begin{equation}\label{eq:deltan}
\delta_n(x)=n\left(\frac{(x+1)\ldots(x+n)\Gamma(x+1)}{n!n^x}-1\right).
\end{equation}
Then we have
\begin{equation}\label{eq:prod}
\prod_{k=1}^n \left(1+\frac{x}{k}\right)=\frac{n^x}{\Gamma(x+1)}\left(1+\frac{\delta_n(x)}{n}\right)
\end{equation}
and the following estimates hold for all $n\in\N$:
\begin{enumerate}
\item[\rm{(i)}]  $|\delta_n(x)|\le 1$ for $-1<x\le 1$
\item[\rm{(ii)}] For any $N\in\N$ there exists a constant $C_N>0$ such that
$$ 
0\le \delta_n(x)\le C_N\;\mbox{for}\; 0\le x\le N.
$$
\end{enumerate}
In particular for $x>-1$
\begin{equation}\label{eq:sim}
\prod_{k=1}^n\left(1+\frac{x}{k}\right)\sim \frac{n^x}{\Gamma(x+1)},\quad n\to\infty.
\end{equation}
\end{lem}

\begin{proof} For $0<x\le 1,n\in\N$ we have the following inequalities
$$
\frac{n^x n!}{x(x+1)\ldots(x+n)}\le \Gamma(x)\le\frac{n^x n!}{x(x+1)\ldots(x+n)}\frac{x+n}{n},
$$
as a consequence of the log-convexity of $\Gamma(x)$. (See
the proof of the Bohr-Mollerup characterization of the Gamma function in \cite[p. 14]{Ar}.) From this inequality we immediately get
\begin{equation}\label{eq:deltan1}
1\le \frac{(x+1)\ldots(x+n)\Gamma(x+1)}{n^x n!}\le 1+\frac{x}{n},
\end{equation}
hence $0\le \delta_n(x)\le x\le 1$, which trivially holds for $x=0$ also.

For $-1<x\le 0$ we apply \eqref{eq:deltan1} to $x+1$ and get
$$
1\le \frac{(x+1)\ldots(x+n)\Gamma(x+1)}{n^x n!}\frac{x+n+1}{n}\le 1+\frac{x+1}{n}.
$$
After multiplication by $n/(n+x+1)$, subtraction of 1 and multiplication by $n$ we get
$$
0\ge \delta_n(x)\ge -\frac{n(x+1)}{n+x+1}\ge -(x+1)\ge -1.
$$
This shows $\rm{(i)}$.

Suppose that $\rm{(ii)}$ holds for some $N\in\N$ and let $N\le x\le N+1$. Applying $\rm{(ii)}$ to
$x-1$ we get
$$
0\le n\left(\frac{(x+1)\ldots(x+n)\Gamma(x+1)}{n^x n!}\frac{n}{x+n}-1\right)\le C_N,
$$
and then
$$
0\le \delta_n(x)-x\le \frac{x+n}{n}C_N,
$$
showing that $0\le \delta_n(x)\le C_{N+1}:=N+1+(N+2)C_N.$
\end{proof}

\begin{rem}\label{thm:extension} {\rm The result of Lemma~\ref{thm:deltan} can be improved, since it can be shown that $\delta_n(x)\to x(x+1)/2$ for $n\to\infty$. More precisely we have
$$
\delta_n(x)=\frac{x(x+1)}{2} + O_x(n^{-1}),
$$
where $nO_x(n^{-1})$ is bounded for $x$ in a compact subset of $]-1,\infty[$ as $n\to\infty$.
}
\end{rem}

The three term recurrence relation
$$
zr_n(z)=b_nr_{n+1}(z)+b_{n-1}r_{n-1}(z)
$$
for a symmetric moment problem 
is satisfied by  $r_n(z)=P_n(z)$ and $r_n(z)=Q_n(z)$, and putting $z=0$ and noting that
$P_{2n+1}(0)=Q_{2n}(0)=0$ we find
\begin{equation}\label{eq:P2n}
P_{2n}(0)=(-1)^n\frac{b_0b_2\ldots b_{2n-2}}{b_1b_3\ldots b_{2n-1}},\quad
Q_{2n+1}(0)=(-1)^n\frac{b_1b_3\ldots b_{2n-1}}{b_0b_2\ldots b_{2n}},\quad n\ge 0,
\end{equation} 
where empty products are defined as 1.

When $(b_n)$ is given by \eqref{eq:bn} we find

\begin{equation}\label{eq:bninc}
\frac{b_{2k-1}^2}{b_{2k}^2}=\prod_{j=1}^p\frac{1+\frac{d_j}{pk}}{1+\frac{e_j}{pk}}.
\end{equation}
This gives
$$
Q_{2n+1}^2(0)=\frac{1}{b_0^2}\prod_{k=1}^n\frac{b_{2k-1}^2}{b_{2k}^2}=
\frac{1}{e_1\ldots e_p}\prod_{k=1}^n
\frac{(1+\frac{d_1}{pk})\ldots(1+\frac{d_p}{pk})}{(1+\frac{e_1}{pk})\ldots(1+\frac{e_p}{pk})}
$$
and
$$
P_{2n}^2(0)=\prod_{k=1}^n\frac{b_{2k-2}^2}{b_{2k-1}^2}=\prod_{k=1}^n
\frac{(1+\frac{e_1-p}{pk})\ldots(1+\frac{e_p-p}{pk})}{(1+\frac{d_1}{pk})\ldots(1+\frac{d_p}{pk})}.
$$
By Lemma~\ref{thm:deltan} and \eqref{eq:DE}  we find for $n\in\N$
\begin{equation}\label{eq:Q}
Q_{2n+1}^2(0)=c_2n^{-(E-D)/p}\prod_{j=1}^p\frac{1+\delta_n(d_j/p)/n}{1+\delta_n(e_j/p)/n},\quad
c_2:=p^{-p}\prod_{j=1}^p\frac{\Gamma(e_j/p)}{\Gamma(1+d_j/p)}
\end{equation}
and
\begin{equation}\label{eq:P}
P_{2n}^2(0)=c_1n^{-[p-(E-D)/p]}\prod_{j=1}^p\frac{1+\delta_n(e_j/p-1)/n}{1+\delta_n(d_j/p)/n},\quad
c_1:=\prod_{j=1}^p\frac{\Gamma(1+d_j/p)}{\Gamma(e_j/p)}.
\end{equation}

Note that 
\begin{equation}\label{eq:constant}
c_1c_2=p^{-p},
\end{equation}
and 
\begin{equation}\label{eq:asymp}
P_{2n}^2(0)\sim c_1n^{-[p-(E-D)/p]},\quad Q_{2n+1}^2(0)\sim c_2n^{-(E-D)/p}.
\end{equation}

We need more precise information about the asymptotic behaviour in \eqref{eq:asymp} and claim the following:

\begin{lem}\label{thm:precas} Given the rates \eqref{eq:birth} and \eqref{eq:death} there exists
a constant $K>0$  and numbers $\tau_n, \rho_n$ such that
\begin{equation}\label{eq:asymp*}
P_{2n}^2(0)= c_1n^{-[p-(E-D)/p]}(1+\tau_n),\quad Q_{2n+1}^2(0)= c_2n^{-(E-D)/p}(1+\rho_n)
\end{equation}
and
\begin{equation}\label{eq:tr} 
|\tau_n|,|\rho_n|\le K/n,\quad n\ge 1.
\end{equation}
\end{lem}

\begin{proof} Define
 $$
\tau_n=\prod_{j=1}^p\frac{1+\delta_n(e_j/p - 1)/n}{1+\delta_n(d_j/p)/n}-1,\quad 
\rho_n=\prod_{j=1}^p\frac{1+\delta_n(d_j/p)/n}{1+\delta_n(e_j/p)/n}-1.
$$
Then \eqref{eq:asymp*} holds. To see that \eqref{eq:tr} holds we write
$$
x_{n,j}=\delta_n(e_j/p - 1),\; y_{n,j}=\delta_n(d_j/p),\quad j=1,\ldots,p
$$
and by Lemma~\ref{thm:deltan} the quantities $x_{n,j},y_{n,j}$ are bounded independent of $n$.
Putting
$$
A_n=n\left(\prod_{j=1}^p (1+x_{n,j}/n) -\prod_{j=1}^p (1+y_{n,j}/n)\right),\quad B_n=\prod_{j=1}^p (1+y_{n,j}/n),
$$
we have $n\tau_n=A_n/B_n$ and
\begin{eqnarray*}
\begin{split}
A_n &=\sum_{j=1}^p (x_{n,j}-y_{n,j})+\frac{1}{n}\sum_{j_1<j_2}(x_{n,j_1}
x_{n,j_2}-y_{n,j_1}y_{n,j_2})\\
&+\ldots + \frac{1}{n^{p-1}}(x_{n,1}\ldots x_{n,p}-y_{n,1}\ldots y_{n,p}).
\end{split}
\end{eqnarray*}
Since $A_n$ is bounded and $B_n\to 1$, we see that 
 $|\tau_n| \le K/n$ for a suitable constant $K>0$ independent of $n$. A similar argument applies to $\rho_n$.
\end{proof}

A moment problem is indeterminate if and only if $(P_n(0)),(Q_n(0))\in\ell^2$, cf. Theorem~\ref{thm:indet}, and by \eqref{eq:asymp} this is equivalent to 
$$
1<(E-D)/p<p-1,
$$
which proves Valent's result in \cite{Va98}: The Stieltjes problem with rates \eqref{eq:birth},\eqref{eq:death}
is indeterminate precisely when  \eqref{eq:valindet} holds.

\section{Results about order and type for symmetric moment problems}

We consider a symmetric  indeterminate Hamburger moment problem given by the three term recurrence relation
$$
zr_n(z)=b_nr_{n+1}(z) + b_{n-1}r_{n-1}(z),
$$ 
for a sequence of positive numbers $b_n>0$. As usual $(P_n)$ and $(Q_n)$ denote the orthonormal polynomials  and those of the second kind.
By the symmetry assumption $P_{2n+1}(0)=Q_{2n}(0)=0$. Therefore 
\eqref{eq:Simon} can be rewritten as
\begin{equation}\label{eq:Simon1}
\begin{pmatrix}
A(z) & B(z)\\
C(z) & D(z)
\end{pmatrix} =\left[\prod_{n=1}^\infty (I+zT_{2n})(I+zT_{2n-1})\right]
\begin{pmatrix}
 0 & -1\\
1 & z
\end{pmatrix}.
\end{equation}
For $n\ge 1$ we introduce
\begin{equation}\label{eq:uv}
v_n=P_{2n}^2(0),\;u_n=Q_{2n-1}^2(0)
\end{equation}
and
\begin{equation}\label{eq:UV}
V_n=\begin{pmatrix}
0 & 0\\ v_n & 0
\end{pmatrix}=-T_{2n},\quad
U_n=\begin{pmatrix}
0 & u_n\\ 0 & 0
\end{pmatrix}=T_{2n-1}.
\end{equation}
We then have

\begin{equation}\label{eq:N-matrix}
\begin{pmatrix}
A(z) & B(z)\\
C(z) & D(z)
\end{pmatrix} =\left[\prod_{n=1}^\infty (I-zV_n)(I+zU_n)\right]
\begin{pmatrix}
 0 & -1\\
1 & z
\end{pmatrix}.
\end{equation}

We observe that $U_mU_n=V_mV_n=0$ and
\begin{equation}\label{eq:UV1}
U_nV_m=\begin{pmatrix} u_nv_m & 0\\ 0 & 0\end{pmatrix},\quad V_nU_m=\begin{pmatrix}
0 & 0\\ 0 & v_nu_m\end{pmatrix}.
\end{equation}
We want to analyse the infinite product and its power series
$$
M(z)=\begin{pmatrix} m_{11}(z) & m_{12}(z)\\m_{21}(z) &m_{22}(z)\end{pmatrix}:=\prod_{n=1}^\infty (I-zV_n)(I+zU_n)=I + \sum_{n=1}^\infty M_nz^n.
$$
Clearly
$$
M_1=\begin{pmatrix} 0 & \sum_{k=1}^\infty u_k\\ -\sum_{k=1}^\infty v_k &0\end{pmatrix}
$$
and for $n\ge 1$
\begin{eqnarray*}
M_{2n}&=&(-1)^n\sum_{1\le k_1\le k_2<\ldots<k_{2n-1}\le k_{2n}} V_{k_{2n}} U_{k_{2n-1}}\cdots V_{k_2}U_{k_1}\\
&+&
(-1)^n\sum_{1\le k_1<k_2\le \ldots\le k_{2n-1}<k_{2n} } U_{k_{2n}}V_{k_{2n-1}}\cdots U_{k_2}V_{k_1},
\end{eqnarray*}
because we shall consider all possible choices of $2n$ parentheses, where we select the term containing $z$. Since the product of two consecutive  $U's$ or $V's$ gives zero, we have to alternate between $U's$ and $V's$, and this also determines the inequalities between the indices.
From \eqref{eq:UV1} we see that $M_{2n}$ is a diagonal matrix.
 There is a similar expression for $M_{2n+1}$  showing that the diagonal elements vanish.

Let us focus on the lower right corner of $M_{2n}$. This matrix entry  is equal to
\begin{equation}\label{eq:a2n}
a_{2n}=(-1)^n\sum_{1\le k_1\le k_2<\ldots<k_{2n-1}\le k_{2n}}u_{k_1}v_{k_2}\cdots u_{k_{2n-1}}v_{k_{2n}},
\end{equation}
and by \eqref{eq:N-matrix} we have
$$
M(z)=\begin{pmatrix}
A(z)z-B(z) & A(z)\\C(z)z-D(z) & C(z)
\end{pmatrix}
$$
hence
$$
m_{22}(z)=1+\sum_{n=1}^\infty a_{2n}z^{2n}=C(z).
$$

Summing up we have proved:

\begin{prop}\label{thm:status} Consider a symmetric indeterminate Hamburger moment problem
and let  $v_n=P_{2n}^2(0), u_n=Q_{2n-1}^2(0)$. The order, type and Phragm{\'e}n-Lindel{\"o}f indicator of the moment problem  is equal to the order, type and Phragm{\'e}n-Lindel{\"o}f indicator of the function
\begin{equation}\label{eq:m22}
m_{22}(z)=1+\sum_{n=1}^\infty a_{2n}z^{2n},
\end{equation}
where $a_{2n}$ is given in \eqref{eq:a2n}.
\end{prop}

\medskip
{\it Proof of Theorem~\ref{thm:main1}.}

 If $\alpha=\beta$ the result follows from \cite[Theorem 4.7]{B:S}.

Suppose next that $\alpha>\beta$.

For $i\le j$ we have
$$
u_iv_j\le Cv_i^{\b/\a}v_j\le CD^{\beta(1/\gamma-1/\a)}v_i^{\beta/\gamma}v_j^{\beta/\gamma},
$$
because the last inequality is equivalent to
$$
v_j^{\beta(1/\b-1/\gamma)}\le (Dv_i)^{\beta(1/\gamma-1/\a)},
$$
which is true by (iii) because $\beta(1/\gamma-1/\a)=\beta(1/\b-1/\gamma)>0$ as a consequence of \eqref{eq:harm}.
Therefore
\begin{equation*}
\begin{split}
|a_{2n}|&=\sum_{1\le k_1\le k_2<\ldots<k_{2n-1}\le k_{2n}}u_{k_1}v_{k_2}\ldots u_{k_{2n-1}}v_{k_{2n}}\\
&\le (CD^{\beta(1/\gamma-1/\a)})^n\sum_{1\le k_1\le k_2<\ldots<k_{2n-1}\le k_{2n}}\left(v_{k_1}v_{k_2}\cdots v_{k_{2n-1}}v_{k_{2n}}\right)^{\beta/\gamma}.
\end{split}
\end{equation*}
The last sum is majorized by the power series coefficient to $z^{2n}$ in 
\begin{equation}\label{eq:F}
F(z)=\prod_{n=1}^\infty\left(1+v_n^{\beta/\gamma}z\right)^2.
\end{equation} 
This shows that
$$
\mathcal M_{m_{22}}(r)=1+\sum_{n=1}^\infty |a_{2n}|r^{2n} \le \mathcal M_F(Kr),
$$
where $K=(CD^{\b(1/\g-1/\a)})^{1/2}$, and therefore $\rho=\rho_{m_{22}}\le \rho_F$, see
the explanation after \eqref{eq:modineq}.

Since $(v_n^{\beta/\gamma})\in\ell^\gamma$, the zeros of $F(z)$ has exponent of convergence $\le\gamma$ and therefore $\rho_F\le \gamma$, cf. \cite{BOAS}, and we conclude that $\rho\le\gamma$.

Suppose finally that $\alpha<\beta$. Introducing $D_u=\sum_{k=1}^\infty u_k,\;D_v=\sum_{k=1}^\infty v_k$
we  get
$$
|a_{2n}|\le D_uD_v \sum_{1\le k_2< k_3\le \ldots\le k_{2n-2}<k_{2n-1}} v_{k_2}u_{k_3}\ldots v_{k_{2n-2}}u_{k_{2n-1}}.
$$

By (ii) and (iii) we similarly get for $i<j$
$$
v_iu_j\le Cv_i v_j^{\b/\a} \le CD^{\beta/\gamma-1} v_i^{\beta/\gamma} v_j^{\beta/\gamma},
$$
so
$$
|a_{2n}|\le D_uD_v (CD^{\beta/\gamma-1})^{n-1} \sum_{1\le k_2< k_3\le \ldots\le k_{2n-2}<k_{2n-1}} \left(v_{k_2}v_{k_3}\ldots v_{k_{2n-2}}v_{k_{2n-1}}\right)^{\beta/\gamma}.
$$
The last sum is majorized by the power series coefficient to $z^{2n-2}$ in $F(z)$ given by \eqref{eq:F}. Again this shows that $\rho\le \gamma$.
$\quad\square$

\begin{lem}\label{thm:orderbelow}
Consider a symmetric  indeterminate moment problem of order $0<\rho<1$ and type $\tau$, and assume that the recurrence coefficients $(b_n)$ satisfy
$$
b_n\le C(n+1)^{\kappa}
$$
 for some $C>0$ and $\kappa$ (necessarily $>1$ by Carleman's criterion).
Then $\rho\ge 1/\kappa$ and if $\rho=1/\kappa$ then $\tau\ge \kappa C^{-1/\kappa}$. 
\end{lem}

\begin{proof} We need the coefficients of the orthonormal polynomials
$$
P_n(x)=\sum_{k=0}^n b_{k,n}x^k 
$$
and the quantity 
$$
c_n=\left(\sum_{j=n}^\infty b_{n,j}^2\right)^{1/2}
$$
studied in \cite[p. 112]{B:S}, where it is proved that
$$
\sum_{n=0}^\infty r^{2n}c_n^2=\frac{1}{2\pi}\int_0^{2\pi} P^2(re^{it})\,dt,\quad r\ge 0.
$$
Here $P$ is the function from Theorem~\ref{thm:indet} and it has order $\rho$ and type $\tau$ by assumption, cf. \cite{B:P}. Using that $c_n>b_{n,n}=1/(b_0\ldots b_{n-1})$,
we get
$$
\sum_{n=0}^\infty\frac{r^{2n}}{b_0^2b_1^2\ldots b_{n-1}^2}\le \frac{1}{2\pi}\int_0^{2\pi} P^2(re^{it})\,dt,
$$
hence by the assumption about $(b_n)$
$$
\sum_{n=0}^\infty\frac{(r/C)^{2n}}{(n!)^{2\kappa}}\le  \frac{1}{2\pi}\int_0^{2\pi} P^2(re^{it})\,dt.
$$
The entire function
$$
\sum_{n=0}^\infty \frac{z^{2n}}{(n!)^{2\kappa}}
$$
has order $1/\kappa$ and type $2\kappa$ by \eqref{eq:deforder} and \eqref{eq:deftype} and the result follows.
\end{proof}

\begin{rem} {\rm In a previous version of Lemma~\ref{thm:orderbelow} we needed the assumption that $b_n$ is eventually increasing. The referee informed us that this condition is not necessary and suggested a variational proof of the simplified result, also indicating that one can obtain results about type in this way. After this remark we realized that our proof could be modified to the present version. 
}
\end{rem}

We now return to the symmetric indeterminate moment problem with Nevanlinna matrix \eqref{eq:N-matrix} and recall the notation
$$
v_n=P_{2n}^2(0),\;u_n=Q_{2n-1}^2(0),\quad n\ge 1.
$$
We assume that
\begin{equation}\label{eq:add}
v_n=c_1n^{-1/\b}(1+\tau_n),\quad u_n=c_2 n^{-1/\a}(1+\rho_n),
\end{equation}
for constants $c_1,c_2>0,0<\a,\b<1$, and assume further
that   there exists a constant $K>0$ such that $|\tau_n|, |\rho_n|\le K/n$ for $n\ge 1$.

\bigskip

Note that by Lemma~\ref{thm:precas} all this is satisfied for the symmetric Hamburger moment problem of Proposition~\ref{thm:sumup}.

\bigskip

For $a\in\N$ consider
\begin{equation}\label{eq:sigmana}
\sigma_n(a)=\sum_{a\le k_1\le k_2<\ldots<k_{2n-1}\le k_{2n}}u_{k_1}v_{k_2}\ldots u_{k_{2n-1}}v_{k_{2n}}
\end{equation}
and 
\begin{equation}\label{eq:sna}
s_n(a)=\sum_{a\le k_1\le k_2<\ldots<k_{2n-1}\le k_{2n}}k_1^{-1/\a}k_2^{-1/\b}\ldots k_{2n-1}^{-1/\a}k_{2n}^{-1/\b}.
\end{equation}

Note that $|a_{2n}|=\sigma_n(1)$,
cf. \eqref{eq:a2n}.

Under these assumptions we have:

\begin{lem}\label{thm:type1}
For all $a\in\N$ one has
\begin{equation}\label{eq:or1}
\lim_{n\to\infty}\root{n}\of {\frac{\sigma_n(1)}{s_n(a)}}=c_1c_2,
\end{equation}
so  the entire functions
\begin{equation}\label{eq:type1}
m_{22}(z)=1+\sum_{n=1}^\infty (-1)^n \sigma_n(1)z^{2n},\quad T_a(z)=\sum_{n=1}^\infty
s_n(a)(c_1c_2)^{n}z^{2n} 
\end{equation}
have the same order and type.
\end{lem}

\begin{proof}
We have
\begin{equation*}
\begin{split} s_n(1) &\ge s_n(a)\\
&=\sum_{1\le k_1\le k_2<\ldots<k_{2n-1}\le k_{2n}}(k_1+a-1)^{-1/\a}(k_2+a-1)^{-1/\b}\ldots (k_{2n}+a-1)^{-1/\b}\\
&=\sum_{1\le k_1\le k_2<\ldots<k_{2n-1}\le k_{2n}}
k_1^{-1/\a}k_2^{-1/\b}\ldots k_{2n}^{-1/\b}\\ &\times
(1+(a-1)/k_1)^{-1/\a}(1+(a-1)/k_2)^{-1/\b}\ldots (1+(a-1)/k_{2n})^{-1/\b}\\
&\ge\sum_{1\le k_1\le k_2<\ldots<k_{2n-1}\le k_{2n}}
k_1^{-1/\a}k_2^{-1/\b}\ldots k_{2n}^{-1/\b}\prod_{j=1}^n(1+(a-1)/j)^{-1/\a-1/\b}
\end{split}
\end{equation*}
since $k_2\ge k_1\ge 1,\;k_4\ge k_3\ge 2,\;\ldots,\;k_{2n}\ge k_{2n-1}\ge n$.

However, by Lemma~\ref{thm:deltan}
$$
\prod_{j=1}^n (1+(a-1)/j)=\frac{n^{a-1}}{\Gamma(a)}\left(1+\frac{\delta_n(a-1)}{n}\right)
\le L(a)n^{a-1}
$$
for some constant $L(a)$ depending on $a$. We therefore get
\begin{equation}\label{eq:sn}
s_n(1)\ge s_n(a)\ge s_n(1)\left(L(a)n^{a-1}\right)^{-2/\gamma},
\end{equation}
showing that for $a\in\N$
\begin{equation}\label{eq:1a}
(s_n(1))^{1/n} \sim (s_n(a))^{1/n},\quad n\to\infty.
\end{equation}

We also have
\begin{eqnarray*}
\begin{split}\sigma_n(1) &\ge \sigma_n(a) \\
 &= (c_1c_2)^n \sum_{a\le k_1\le k_2<\ldots<k_{2n-1}\le k_{2n}} k_1^{-1/\a}k_2^{-1/\b}\ldots k_{2n-1}^{-1/\a}k_{2n}^{-1/\b}\\
&\times (1+\rho_{k_1})(1+\tau_{k_2})\ldots (1+\rho_{k_{2n-1}})(1+\tau_{k_{2n}}).
\end{split}
\end{eqnarray*}
Given $\eps>0$ there exists $a\in\N$ such that $|\tau_k|,|\rho_k|\le \eps$ for $k\ge a$.
For such $a$ we then have
$$
\sigma_n(1)\ge\sigma_n(a)\ge (c_1c_2)^{n}s_n(a)(1-\eps)^{2n},
$$
so by \eqref{eq:sn}
$$
\sigma_n(1)\ge (c_1c_2)^{n}s_n(1)\left(L(a)n^{a-1}\right)^{-2/\gamma}(1-\eps)^{2n}.
$$
Since $\eps>0$ is arbitrary we find
\begin{equation}\label{eq:liminf}
\liminf_{n\to\infty}\left(\frac{\sigma_n(1)}{s_n(1)}\right)^{1/n}\ge c_1c_2.
\end{equation}

On the other hand we have
\begin{eqnarray*}
\begin{split}\sigma_n(1)& 
 \le (c_1c_2)^{n} \sum_{1\le k_1\le k_2<\ldots<k_{2n-1}\le k_{2n}} k_1^{-1/\a}k_2^{-1/\b}\ldots k_{2n-1}^{-1/\a}k_{2n}^{-1/\b}\\
&\times (1+K/k_1)(1+K/k_2)\ldots (1+K/k_{2n-1})(1+K/k_{2n})\\
&\le (c_1c_2)^{n}s_n(1)\prod_{j=1}^n(1+K/j)^2=(c_1c_2)^{n}s_n(1)\left(\frac{n^K}{\Gamma(K+1)}
(1+\delta_n(K)/n)\right)^2,
\end{split}
\end{eqnarray*}
where we have used Lemma~\ref{thm:deltan}.

We then get
$$
\limsup_{n\to\infty} \left(\frac{\sigma_n(1)}{s_n(1)}\right)^{1/n}\le c_1c_2,
$$
which together with \eqref{eq:liminf} proves that
\begin{equation}\label{eq:nthroot}
\sigma_n(1)^{1/n}\sim s_n(1)^{1/n}c_1c_2.
\end{equation}
Combining \eqref{eq:1a} and \eqref{eq:nthroot} we get \eqref{eq:or1}.
Using \eqref{eq:deforder} and \eqref{eq:deftype}
it follows from \eqref{eq:or1}  that the functions \eqref{eq:type1}
have the same order and type.
\end{proof}

Under the same assumptions as in Lemma~\ref{thm:type1} we have:

\begin{lem}\label{thm:more1} For $n\in\N$ define
\begin{equation}\label{eq:gamman}
\g_n=\sum_{1\le k_1\le k_2<\ldots<k_{2n-1}\le k_{2n}}\left(k_1k_2\ldots k_{2n-1}k_{2n}\right)^{-1/\gamma},
\end{equation}
where $\g$ is the harmonic mean of $\a$ and $\b$, i.e., $\g_n=\g_n(2/\g)$ from \eqref{eq:mz}.
Then
\begin{equation}\label{eq:or2}
\root{n}\of{\g_n}\sim \root{n}\of{s_n(1)}
\end{equation} 
so the entire functions
$$
H(z)=\sum_{n=1}^\infty s_n(1)z^{2n},\quad G(z)=\sum_{n=1}^\infty \gamma_nz^{2n}
$$
have the same order and type.
\end{lem}

\begin{proof} If $\a=\b$ there is nothing to prove.

Suppose next that $\a>\b$, hence $1/\a<1/\gamma<1/\b$.
From the inequalities
$$
k_{2j-1}^{-1/\a}k_{2j}^{-1/\b}\le k_{2j-1}^{-1/\g}k_{2j}^{-1/\g},\quad j=1,\ldots,n,
$$
which hold because 
$$
k_{2j-1}^{1/\g-1/\a}\le k_{2j}^{1/\b-1/\g},\quad j=1,\ldots,n,
$$
we get $s_n(1)\le \g_n$.

We also have
\begin{eqnarray*}
s_n(1) &\ge& \zeta(1/\a)^{-1} \sum_{1\le k_1\le k_2<\ldots<k_{2n-1}\le k_{2n}<k_{2n+1}}k_1^{-1/\a}k_2^{-1/\b}\ldots k_{2n-1}^{-1/\a}k_{2n}^{-1/\b}k_{2n+1}^{-1/\a}\\
&\ge& \zeta(1/\a)^{-1} \sum_{1\le k_2<\ldots<k_{2n-1}\le k_{2n}<k_{2n+1}}k_2^{-1/\b}\ldots k_{2n-1}^{-1/\a}k_{2n}^{-1/\b}k_{2n+1}^{-1/\a},
\end{eqnarray*}
where we have added the extra summation variable $k_{2n+1}$ and then put $k_1=1$. Using
$$
k_{2j}^{-1/\b}k_{2j+1}^{-1/\a}\ge k_{2j}^{-1/\g}k_{2j+1}^{-1/\g},\quad j=1,\ldots,n,
$$
we get
$$
s_n(1)\ge \zeta(1/\a)^{-1} \sum_{1\le k_2<\ldots<k_{2n-1}\le k_{2n}<k_{2n+1}}\left(k_2k_3\ldots k_{2n}k_{2n+1}\right)^{-1/\g}.
$$
Applying the same procedure once more we get
$$
s_n(1)\ge (\zeta(1/\a)\zeta(1/\g))^{-1} \sum_{1< k_3\le k_4<\ldots< k_{2n+1}\le k_{2n+2}}\left(k_3\ldots k_{2n+2}\right)^{-1/\g}.
$$
Substituting $k_{j+2}=l_j+1,j=1,\ldots, 2n$ and then renaming $l_j$ to $k_j$ we get
\begin{eqnarray*}
s_n(1) &\ge & (\zeta(1/\a)\zeta(1/\g))^{-1} \sum_{1\le k_1\le k_2<\ldots< k_{2n-1}\le k_{2n}}\left((k_1+1)\ldots (k_{2n}+1)\right)^{-1/\g}\\
&=& (\zeta(1/\a)\zeta(1/\g))^{-1} \sum_{1\le k_1\le k_2<\ldots< k_{2n-1}\le k_{2n}}(k_1\ldots k_{2n})^{-1/\g}\\
 &\times &\left((1+1/k_1)\ldots (1+1/k_{2n})\right)^{-1/\g}\\
&\ge & (\zeta(1/\a)\zeta(1/\g))^{-1} (4n^2)^{-1/\g}\gamma_n,
\end{eqnarray*}

where we have used that
$$
\prod_{j=1}^n(1+1/k_{2j-1})(1+1/k_{2j})\le \prod_{j=1}^n (1+1/j)^2=(n+1)^2\le 4n^2.
$$
Summing up we have
\begin{equation}\label{eq:sumup1}
 (\zeta(1/\a)\zeta(1/\g))^{-1} (4n^2)^{-1/\g}\gamma_n\le s_n(1)\le \g_n,
\end{equation}
and \eqref{eq:or2} follows.

\medskip
Suppose finally that $\a<\b$, hence $1/\b<1/\g<1/\a$.

We then get
\begin{eqnarray*}
s_n(1) &\le& \zeta(1/\a)\zeta(1/\b) \sum_{1\le k_2< k_3\le\ldots< k_{2n-1}}k_2^{-1/\b}k_3^{-1/\a}\ldots k_{2n-1}^{-1/\a}\\
&\le& \zeta(1/\a)\zeta(1/\b) \sum_{1\le k_2< k_3\le\ldots< k_{2n-1}}(k_2 k_3\ldots k_{2n-1})^{-1/\g}
\end{eqnarray*}
because
$$
k_{2j}^{-1/\b}k_{2j+1}^{-1/\a} \le (k_{2j}k_{2j+1})^{-1/\g},\quad j=1,\ldots,n-1.
$$
Applying the same procedure once more we get
\begin{eqnarray*}
s_n(1) &\le& \zeta(1/\a)\zeta(1/\b)\zeta^2(1/\g) \sum_{1< k_3\le k_4<\ldots< k_{2n-3}\le k_{2n-2}}(k_3\ldots k_{2n-2})^{-1/\g}\\
&\le& \zeta(1/\a)\zeta(1/\b)\zeta^2(1/\g) \g_{n-2}.
\end{eqnarray*}
On the other hand, it is easy to see, by the same methods as before, that $s_n(1)\ge \g_n$, so summing up  we have
\begin{equation}\label{eq:sumup2}
 \g_n\le s_n(1)\le \zeta(1/\a)\zeta(1/\b)\zeta^2(1/\g) \g_{n-2},
\end{equation}
and \eqref{eq:or2} follows.
Using \eqref{eq:deforder} and \eqref{eq:deftype} it follows that the two functions have the same order and type.
\end{proof}

\section{On Valent's conjecture about order and type}

In this section we shall apply our results about symmetric moment problems.

\medskip
{\it Proof of Theorem~\ref{thm:valent-order}.}

 By Proposition~\ref{thm:order/type} it suffices to prove that the order $\rho$ of the symmetric Hamburger moment problem of Proposition~\ref{thm:sumup} is equal to $2/p$.

By \eqref{eq:asymp} we have
$$
v_n\sim c_1n^{-1/\b},\quad u_n\sim c_2 n^{-1/\a}
$$
with
$$
\b=(p-(E-D)/p)^{-1},\quad \a=p/(E-D),
$$
and $c_1,c_2$ are given by \eqref{eq:P},\eqref{eq:Q}. The harmonic mean of $\a$ and $\b$ is $\gamma=2/p$.

For $\eps>0$ we define $\a_\eps=\a+\eps, \b_\eps=(1+\eps/\a)\b$ and see that
$(u_n)\in\ell^{\a_\eps},(v_n)\in\ell^{\b_\eps}$. Note that $\b_\eps/\a_\eps=\b/\a$, so the condition (ii) of Theorem~\ref{thm:main1} is satisfied for $\a_\eps,\b_\eps$, while (iii) holds because of the asymptotic  behaviour of $(v_n)$.

 We then get that
$\rho\le \gamma_\eps$,  where $\gamma_\eps$ is the harmonic mean of $\a_\eps$ and $\b_\eps$. Letting $\eps\to 0$ we get $\rho\le \gamma=2/p$.
That $\rho\ge 2/p$ follows from Lemma~\ref{thm:orderbelow} because
the recurrence coefficients $(b_n)$ given by \eqref{eq:bn} have the asymptotic behaviour
$$
b_n\sim (p/2)^{p/2} n^{p/2}.
$$ 
$\quad\square$

\begin{rem}\label{thm:remx} {\rm The type part of Lemma~\ref{thm:orderbelow} can be applied to yield $\tau\ge (p/2)C^{-2/p}$ for any $C>(p/2)^{p/2}$, hence $\tau\ge 1$. This estimate is however weaker than the lower estimate of Theorem~\ref{thm:valenttype}.
}
\end{rem}

The special case $p=3$ and $(e_j)=(1,2,2),(d_j)=(0,0,1)$ of the rates \eqref{eq:birth}, \eqref{eq:death} was considered in \cite{Va98} and with more details in \cite{G:L:V}.
In this case
$$
\lambda_n<\sqrt{\mu_n\mu_{n+1}},n\ge 2,\quad \mu_{n+1}>\sqrt{\lambda_n\lambda_{n+1}},n\ge 1
$$
so $(b_n)$ in \eqref{eq:bn} is alternately log-convex and log-concave.

There is a special case where the  condition of log-concavity is true, and in this case the Valent conjecture concerning order follows from the results in \cite{B:S}.

\begin{prop}\label{thm:reg} Consider the special case of the rates \eqref{eq:birth}, \eqref{eq:death}, where $d_j=e_j-p/2,\; j=1,\ldots,p$, $p\ge 3$. Then the sequence $(b_n)$ in
\eqref{eq:bn} is eventually log-concave.
\end{prop}

\begin{proof} Note that $E-D=p^2/2$ and therefore \eqref{eq:valindet} holds.
The function
$$
f(x)=\prod_{j=1}^p(px+e_j)
$$
satisfies $f(n)=\lambda_n$ and $f(n-1/2)=\mu_n$, and since $f(x)$ is log-concave for $x>-e_1/p$, it follows that $(b_n)$ is log-concave. Using that
$$
b_n\sim (p/2)^{p/2}n^{p/2},
$$
we see that the exponent of convergence of $(b_n)$ is $2/p$, so the symmetric Hamburger moment problem has order $2/p$ by Theorem~\ref{thm:beralpha1}. 
\end{proof}

The special case $p=4$ and $(e_j)=(1,2,2,3),(d_j)=(-1,0,0,1)$ satisfies the conditions of Proposition~\ref{thm:reg} and was studied in \cite{B:V}.

\medskip

{\it Proof of Theorem~\ref{thm:main2}.}

We apply Lemma~\ref{thm:type1} and Lemma~\ref{thm:more1} to the symmetric moment problem 
from Proposition~\ref{thm:sumup}.
In this case
$$
\b=(p-(E-D)/p)^{-1},\quad \a=p/(E-D),\quad c_1c_2=p^{-p},\quad \g=2/p.
$$ 
By Theorem~\ref{thm:valent-order} we see that the common order of the functions \eqref{eq:type1} is $\rho=2/p.$

Since $T_1(z)=H((c_1c_2)^{1/2}z)$  the type $\tau$ of the functions \eqref{eq:type1} is 
$$
\tau=\tau_{T_1}=(c_1c_2)^{\rho/2}\tau_H=\tau_H/p.
$$
Finally, by Lemma~\ref{thm:more1} we have that $\tau=\tau_H/p=\tau_G/p$.

The type of the Stieltjes problem with rates \eqref{eq:birth} and \eqref{eq:death} is  by Proposition~\ref{thm:order/type}   equal to $\tau_H/p=\tau_G/p$. Finally note that $G_p$ given by \eqref{eq:Gp} and $G(z)=\sum_{n=1}\gamma_nz^{2n}$ from Lemma~\ref{thm:more1} have orders and types
$\rho_G=2\rho_{G_p}$ and $\tau_G=\tau_{G_p}$.
$\quad\square$

\medskip
{\it Proof of Theorem~\ref{thm:valenttype}.}  In the definition \eqref{eq:mz} of $\g_n(p)$ we choose $k_1=k_2=l_1,k_3=k_4=l_2,\ldots,k_{2n-1}=k_{2n}=l_n$ and get
$$
\g_n(p)>\sum_{1\le l_1<l_2<\ldots<l_n}(l_1l_2\ldots l_n)^{-p},
$$
showing that for $r>0$
\begin{equation}\label{eq:tv}
\prod_{n=1}^\infty \left(1+\frac{r^2}{n^p}\right)<1+\sum_{n=1}^\infty \g_n(p)r^{2n}<
\prod_{n=1}^\infty \left(1+\frac{r}{n^{p/2}}\right)^2.
\end{equation}
The three expressions above are the maximum moduli for three entire functions of the same order $2/p$, since it is known 
from \cite{BOAS} that the canonical product
$$
c(z)=\prod_{n=1}^\infty \left(1+\frac{z}{n^{1/\rho}}\right),\quad 0<\rho<1,
$$
has order $\rho_c=\rho$ and type $\tau_c=\pi/\sin(\pi\rho)$. From \eqref{eq:tv} we then get the following inequalities between the types of the three functions
$$
\frac{\pi}{\sin(\pi/p)}\le \tau_p \le \frac{2\pi}{\sin(2\pi/p)}=\frac{\pi}{\sin(\pi/p)\cos(\pi/p)},
$$
and dividing with $p$ yields the inequality \eqref{eq:type}. $\quad\square$

\medskip

Valent conjectured that the type is
$$
T=\int_0^1\frac{du}{(1-u^p)^{2/p}}
=\frac{1}{p}B(1/p,1-2/p)
$$
We have
\begin{equation}\label{eq:1}
T=\frac{(1/p)\Gamma(1/p)\Gamma(1-2/p)}{\Gamma(1-1/p)}=\frac{\pi/p}{\sin(\pi/p)}
\frac{\Gamma(1-2/p)}{\Gamma^2(1-1/p)},
\end{equation}
where we have used Euler's reflection formula for $\Gamma$.

We claim that $T$ lies between the bounds for the type $\tau$ given in \eqref{eq:type}.

The upper bound comes from
$$
T<\int_0^1\frac{du}{(1-u^{p/2})^{2/p}}=\frac{2\pi}{p\sin(2\pi/p)},
$$
where the integral is evaluated by transforming it to a Beta-integral and using the reflection formula.

Using \eqref{eq:1} the lower bound is equivalent to
$$
\Gamma^2(1-1/p)\le \Gamma(1-2/p),\quad p=3,4,\ldots,
$$
which follows from the log-convexity of $\Gamma$:
$$
\Gamma(\alpha x+(1-\alpha)y)\le \Gamma^{\alpha}(x)\Gamma^{1-\alpha}(y)
$$
for $\alpha=1/2$, $x=1,y=1-2/p$.

\medskip
{\it Proof of Theorem~\ref{thm:main3}.}  We have
\begin{eqnarray*}
\gamma_n(p)&\le& \sum_{1\le k_1<k_3<\ldots<k_{2n-3}} (k_1k_3\ldots k_{2n-3})^{-p}
(k_3-k_1)\ldots(k_{2n-3}-k_{2n-5})\\
&\times& \sum_{k_{2n-3}<k_{2n-1}\le k_{2n}}(k_{2n-1}k_{2n})^{-p/2},
\end{eqnarray*}
where we have used that 
$$
(k_{2j-1}k_{2j})^{-p/2}\le k_{2j-1}^{-p}
$$ 
for the $k_{2j+1}-k_{2j-1}$ values of $k_{2j}\in [k_{2j-1},k_{2j+1}[$, $j=1,\ldots,n-2$.

We next use
$$
\sum_{k=a}^\infty k^{-p/2}\le \int_{a-1}^\infty x^{-p/2}\,dx=\frac{(a-1)^{-p/2+1}}{p/2-1}
\le \frac{2^{p/2-1}}{p/2-1}a^{-p/2+1},
$$
where  $a\ge 2$, hence $a-1\ge a/2$, 
so for  $a=k_{2n-1}$ we get
$$
\gamma_n(p)\le \frac{2^{p/2-1}}{p/2-1}\sum_{1\le k_1<k_3<\ldots<k_{2n-1}} (k_1k_3\ldots k_{2n-1})^{-p}
(k_3-k_1)\ldots(k_{2n-3}-k_{2n-5})k_{2n-1}
$$
and by changing symbols
\begin{eqnarray*}
\begin{split}
\gamma_n(p) &\le \frac{2^{p/2-1}}{p/2-1}\sum_{1\le k_1<k_2<\ldots<k_{n}} (k_1k_2\ldots k_{n})^{-p}
(k_2-k_1)\ldots(k_{n-1}-k_{n-2})k_{n}\\
&\le  \zeta(p-1)\frac{2^{p/2-1}}{p/2-1}\sum_{1\le k_1<k_2<\ldots<k_{n-1}} (k_1k_2\ldots k_{n-1})^{-p}
(k_2-k_1)\ldots(k_{n-1}-k_{n-2}).
\end{split}
\end{eqnarray*}
This shows the right-hand side of \eqref{eq:mz3}.

On the other hand we have
$$
\gamma_n(p)\ge \sum_{1\le k_2<k_4<\ldots<k_{2n}}(k_2k_4\ldots k_{2n})^{-p}k_2(k_4-k_2)
\ldots(k_{2n}-k_{2n-2}),
$$
where we have used that 
$$
(k_{2j-1}k_{2j})^{-p/2}\ge k_{2j}^{-p}
$$ 
for the $k_{2j}-k_{2j-2}$ values of $k_{2j-1}\in ]k_{2j-2},k_{2j}]$, $j=2,\ldots,n$, respectively
$k_2$ values of $k_1$ when $j=1$.

By changing symbols we can rewrite the last inequality as
$$
\gamma_n(p)\ge \sum_{1\le k_1<k_2<\ldots<k_{n}}(k_1k_2\ldots k_{n})^{-p}k_1(k_2-k_1)
\ldots(k_{n}-k_{n-1}),
$$
which shows the left-hand side of \eqref{eq:mz3}. $\quad\square$

\medskip
{\it Proof of Theorem~\ref{thm:last}} From the proof of Proposition~\ref{thm:finch}
we see that changing $(b_n)$ from $n^c$ to $(n+1)^c$ will not change the order and type.
 For $b_n=(n+1)^c$ we find from \eqref{eq:P2n}
$$
P_{2n}^2(0)=\left(\frac{(2n)!}{2^{2n}(n!)^2}\right)^{2c},\quad 
Q_{2n+1}^2(0)=\left(\frac{2^{2n}(n!)^2}{(2n+1)!}\right)^{2c}.
$$
Using Stirling's formula we then get
$$
v_n=\pi^{-c}n^{-c}(1+O(1/n)),\quad u_n=(\pi/4)^cn^{-c}(1+O(1/n)),
$$
so we see that $v_n,u_n$ satisfy the conditions of \eqref{eq:add} with
$\a=\b=\g=1/c$ and $c_1c_2=2^{-2c}$.

From Lemma~\ref{thm:type1} we see that the type  of the problem is equal to the type of
the function
$$
T_1(z)=\sum_{n=1}^\infty s_n(1)(2^{-c}z)^{2n},
$$
where $s_n(1)=\g_n(2c)$. Since $T_1(z)$ has order $1/c$, we see that the type of the function $\sum \g_n(2c)z^{2n}$ is the double of the type of $T_1(z)$, and the result follows. $\quad\square$

\medskip
{\bf Acknowledgment} The authors want to thank the referee for a very thorough report which has lead to important improvements of the presentation.

Christian Berg; email:berg@math.ku.dk\\
Department of Mathematical Sciences, University of Copenhagen,\\
Universitetsparken 5, DK-2100, Denmark

\medskip
Ryszard Szwarc; email szwarc2@gmail.com\\
Institute of Mathematics,
University of Wroc{\l}aw, pl.\ Grunwaldzki 2/4, 50-384 Wroc{\l}aw, Poland 
 
\end{document}